\documentclass[12pt,thmsa]{article}
\usepackage{amsmath, latexsym, amsfonts, amssymb, amsthm, amscd}
\usepackage{color}

\newtheorem{theorem}{Theorem}
\newtheorem{corollary}{Corollary}

\newtheorem{rem}{Remark}

\newcommand{\p}{\Bbb{P}}

\newcommand{\e}{\Bbb{E}}

\newcommand{\R}{\mathbb{R}}

\newcommand{\ud}{\mathrm{d}}

\newcommand{\reals}{\mathbb{R}}
\newcommand{\ind}{\mathbf{1}}

\newcommand{\exptime}{\mathbf{e}_q}
\newcommand{\parisiantime}{\tau_q}

\newcommand{\qscale}{W^{(q)}}
\newcommand{\pscale}{W^{(p)}}

\title{\textbf{Gerber--Shiu distribution at Parisian ruin for L\'evy insurance risk processes}}

\author{\textbf{E.J. Baurdoux\footnote{Department of Statistics, London School of Economics. Houghton street, {\sc London, WC2A 2AE, United Kingdom.} Email: e.j.baurdoux@lse.ac.uk}\, , J.C. Pardo\footnote{Centro de Investigaci\'on en Matem\'aticas A.C. Calle Jalisco s/n. C.P. 36240, {\sc Guanajuato, Mexico.}  Email: jcpardo@cimat.mx. Corresponding author.}\,\,, J.L. P\'erez\footnote{Department of Probability and Statistics, IIMAS, UNAM. , C.P. 04510 {\sc Mexico, D.F., Mexico.} Email: garmendia@sigma.iimas.unam.mx}\, , J.-F. Renaud\footnote{D\'epartement de math\'ematiques, Universit\'e du Qu\'ebec \`a Montr\'eal (UQAM). 201 av.\ Pr\'esident-Kennedy, {\sc Montr\'eal (Qu\'ebec) H2X 3Y7, Canada.} Email: renaud.jf@uqam.ca}}}

\date{\footnotesize This version: \today}

\begin{document}

\maketitle

\begin{abstract}
\bigskip

Inspired by works of Landriault et al. \cite{LRZ-0, LRZ}, we study the Gerber--Shiu distribution at Parisian ruin with exponential implementation delays for a spectrally negative L\'evy insurance risk process. To be more specific, we study the so-called Gerber--Shiu distribution for a ruin model where at each time the surplus process goes negative, an independent exponential clock is started. If the clock rings before the  surplus becomes positive again then the insurance company is ruined. Our methodology uses excursion theory for spectrally negative L\'evy processes and relies on the theory of so-called scale functions. In particular, we extend recent results of Landriault et al. \cite{LRZ-0, LRZ}.  

\bigskip

\end{abstract}

\vspace{0.5cm}

\section{Introduction and main results}

Originally motivated by  pricing  American claims, Gerber and Shiu \cite{GS, GS1} introduced in risk theory a function that jointly penalizes the present value of the time of ruin, the surplus before ruin and the deficit after ruin for Cram\'er--Lundberg-type processes.
Since then this expected discounted penalty function, by now known as the Gerber--Shiu function, has been deeply studied. Recently, Biffis and Kyprianou \cite{BK} characterized a generalized version of this function in the setting of processes with stationary and independent increments with no positive jumps, also known as spectrally negative L\'evy processes, using scale functions. In the current actuarial setting, we refer to the latter class of processes as L\'evy insurance risk processes.

In the traditional ruin theory literature, if the surplus becomes negative, the company is ruined and has to go out of business. Here, we distinguish between {\it being ruined} and {\it going out of business}, where the probability of {\it going out of business} is a function of the level of negative surplus. The idea of this notion of {\it going out of business} comes from the observation that in some industries, companies can continue doing business even though they are technically ruined (see \cite{LRZ-0} for more motivation). In this paper, our definition of {\it going out of business} is related to so-called Parisian ruin. The idea of this type of actuarial ruin has been introduced by A.\ Dassios and S.\ Wu \cite{DW}, where they consider the application of an implementation delay in the recognition of an insurer's capital insufficiency. More precisely, they assume that ruin occurs if the excursion below the critical threshold level is longer than a deterministic time. It is worth pointing out that this definition of ruin is referred to as Parisian ruin due to its ties with Parisian options (see Chesney et al. \cite{CJY}).

In \cite{DW}, the analysis of the probability of Parisian ruin is done in the context of the classical Cram\'er--Lundberg model. More recently, Landriault et al. \cite{LRZ-0,LRZ} and Loeffen et al. \cite{L} considered the idea of Parisian ruin with respectively a stochastic implementation delay and a deterministic implementation delay, but in the more general setup of L\'evy insurance risk models. In \cite{LRZ-0}, the authors assume that the deterministic delay is replaced by a stochastic grace period with a pre-specified distribution, but they restrict themselves to the study of a L\'evy insurance risk process with paths of bounded variation; explicit results are obtained in the case the delay is exponentially distributed. The model with a deterministic delay has also been studied in the L\'evy setup by Czarna and Palmowski \cite{CP} and by Czarna \cite{C}.

In this paper, we study the Gerber--Shiu distribution at Parisian ruin for general L\'evy insurance risk processes, when the implementation delay is exponentially distributed. Since the L\'evy insurance risk process does not jump at the time when Parisian ruin occurs, the Gerber--Shiu function that we present here only considers the discounted value of the surplus at ruin. Our results extend those of Landriault et al. \cite{LRZ-0}, in the exponential case, by simultaneously considering more general ruin-related quantities and L\'evy insurance risk processes of unbounded and bounded variation. Our approach is based on a \textit{heuristic} idea presented in \cite{LRZ} and which consists in \textit{marking} the excursions away from zero of the underlying surplus process. We will fill this gap and provide a rigorous definition of the time of Parisian ruin. Our main contribution is an explicit and compact expression, expressed in terms of the scale functions of the process, for the Gerber--Shiu distribution at Parisian ruin. From our results, we easily deduce the probability of Parisian ruin originally obtained by Landriault et al. \cite{LRZ-0,LRZ}.

The rest of the paper is organized as follows. In the remainder of Section 1, we introduce  L\'evy insurance risk processes and their associated scale functions and we state some well-known fluctuation identities that will be useful for the sequel. We also introduce, formally speaking,  the notion of Parisian ruin in terms of the excursions away from $0$ of the L\'evy insurance risk process and we provide the main results of this paper. As a consequence, we recover the results that appear in Landriault et al. \cite{LRZ-0,LRZ} and remark on an interesting link with recent findings in \cite{AIZ} on exit identities of spectrally negative L\'evy processes observed at Poisson arrival times. Section 2 is devoted to the proofs of the main results.

\subsection{L\'evy insurance risk processes}
In what follows, we assume that  $X=(X_t, t\geq 0)$ is  a spectrally negative L\'evy process  with no monotone paths (i.e. we exclude the case of the negative of a subordinator) defined on a  probability space $(\Omega, \mathcal{F}, \p)$.  For $x\in \R$ denote by $\p_x$ the law of $X$ when it is started at $x$ and write for convenience  $\p$ in place of $\p_0$. Accordingly, we shall write $\e_x$ and $\e$ for the associated expectation operators.  It is well known that the Laplace exponent $\psi:[0,\infty) \to \R$ of $X$, defined by
\[
\psi(\lambda):=\log\e\Big[{\rm e}^{\lambda X_1}\Big], \quad  \lambda\ge 0,
\]
 is given by the so-called L\'evy-Khintchine formula
\begin{equation}
\psi(\lambda)=\gamma\lambda+\frac{\sigma^2}{2}\lambda^2-\int_{(0, \infty)}\big(1-{\rm e}^{-\lambda x}-\lambda x\mathbf{1}_{\{x<1\}}\big)\Pi(\ud x),\notag
\end{equation}
where $\gamma\in \R$, $\sigma\ge 0$ and $\Pi$ is a measure on $(0,\infty)$  satisfying
\[
\int_{(0,\infty)}(1\land x^2)\Pi(\ud x)<\infty,
\]
which is called the L\'evy measure of $X$. Even though X only has negative jumps, for convenience we choose the L\'evy measure to have only mass on the positive
instead of the negative half line. 

It is also known that $X$ has paths of bounded variation if and only if 
\[
\sigma=0 \quad  \textrm{and} \quad \int_{(0, 1)} x \, \Pi(\mathrm{d}x)<\infty.
\]
 In this case $X$ can be written as
$X_t=ct-S_t,$ $t\geq 0$,
where $c=\gamma+\int_{(0,1)} x\Pi(\mathrm{d}x)$ and $(S_t,t\geq0)$ is a driftless subordinator. Note that  necessarily $c>0$, since we have ruled out the case that $X$ has monotone paths. In this case its Laplace exponent is given by
\begin{equation*}
\psi(\lambda)= \log \mathbb{E} \left[ \mathrm{e}^{\lambda X_1} \right] = c \lambda-\int_{(1,\infty)}\big(1- {\rm e}^{-\lambda x}\big)\Pi(\ud x).\notag
\end{equation*}
The reader is referred to Bertoin \cite{B} and Kyprianou \cite{K} for a complete introduction to the theory of L\'evy processes.

A key element of the forthcoming analysis relies on the theory of so-called scale functions for spectrally negative L\'evy processes. We therefore devote some time in this section reminding the reader of some fundamental properties of scale functions.  For
each $q\geq0$, define  $W^{(q)}:
\R\to [0, \infty),$ such that $W^{(q)}(x)=0$ for all $x<0$ and on $[0,\infty)$ is the unique continuous function with Laplace transform
\begin{eqnarray}
\int^{\infty}_0\mathrm{e}^{-\lambda x}W^{(q)}(x)\mathrm{d}x=\frac1{\psi(\lambda)-q}, \quad \lambda>\Phi(q),\notag
\end{eqnarray}
where $ \Phi(q) = \sup\{\lambda \geq 0: \psi(\lambda) = q\}$ which is well defined and finite for all $q\geq 0$, since $\psi$ is a strictly convex function satisfying $\psi(0) = 0$ and $\psi(\infty) = \infty$. The initial value of $W^{(q)}$ is known to be
\begin{equation*}
W^{(q)}(0)=
\begin{cases}
1/c & \text{when $\sigma=0$ and $\int_{(0, 1)} x \, \Pi(\mathrm{d}x)<\infty$},  \\
0 & \text{otherwise},
\end{cases}
\end{equation*}
where we used the following definition: $W^{(q)}(0) = \lim_{x \downarrow 0} W^{(q)}(x)$. For convenience, we write $W$ instead of $W^{(0)}$.
Associated to the functions $W^{(q)}$ are the functions $Z^{(q)}:\R\to[1,\infty)$ defined by
\[
Z^{(q)}(x)=1+q\int_{0}^x W^{(q)} (y)\ud y,\qquad q\ge 0.
\]
Together, the functions $W^{(q)}$ and $Z^{(q)}$ are collectively known as $q$-scale functions and predominantly appear in almost all fluctuation identities for spectrally negative L\'evy processes.


The theorem below is a collection of known fluctuation identities which will be used throughout this work. See, for example, Chapter 8 of \cite{K} for proofs and the origin of these identities.
\begin{theorem}\label{fi}
Let $X$ be a spectrally negative L\'evy process and denote (for $a>0$) the first passage times by
\begin{equation}
\tau_a^+=\inf\{t>0:X_t>a\}\qquad\text{and}\qquad\tau_0^-=\inf\{t>0:X_t<0\}.\notag
\end{equation}
\begin{itemize}
\item[(i)] For $q\geq 0$ and $x\leq a$
\begin{equation}\label{fi2}
\mathbb{E}_x\Big[ \mathrm e^{-q\tau_a^+} \mathbf{1}_{\{\tau_0^->\tau_a^+\}}\Big]=\frac{W^{(q)}(x)}{W^{(q)}(a)}.
\end{equation}
\item[(ii)] For any $a>0,x,y\in[0,a],q\geq 0$
\begin{equation}\label{fi4}
\int_0^{\infty} \mathrm e^{-q t}\mathbb{P}_x \left( X_t\in \mathrm dy,t<\tau_a^+\wedge\tau_0^- \right) \mathrm{d}t = \left\{\frac{W^{(q)}(x)W^{(q)}(a-y)}{W^{(q)}(a)}-W^{(q)}(x-y)\right\}\mathrm{d}y.
\end{equation}
\end{itemize}
\end{theorem}

Finally, we recall the following two useful identities taken from \cite{LoRZ}: for $p,q \geq 0$ and $x \in \reals$, we have 
\begin{equation}\label{eq:scale_sym}
(q-p)\int_0^x\pscale(x-y)\qscale(y)\mathrm{d}y=\qscale(x)-\pscale(x)
\end{equation}
and, for $p,q \geq 0$ and $y \leq a \leq x \leq b$, we have (with the obvious notation for $\tau_a^-$) 
\begin{multline}\label{eq:exp_scale}
\mathbb E_x \left[  \mathrm e^{-p\tau_a^-} W^{(q)}(X_{\tau_a^-} -y) \mathbf{1}_{\{\tau_a^-<\tau_b^+\}} \right] = W^{(q)}(x-y) - (q-p) \int_a^x W^{(p)}(x-z) W^{(q)}(z-y)\mathrm{d}z \\
- \frac{W^{(p)}(x-a)}{W^{(p)}(b-a)} \left( W^{(q)}(b-y) - (q-p) \int_a^b W^{(p)}(b-z) W^{(q)}(z-y) \mathrm{d}z \right) .
\end{multline}
%

\subsection{Parisian ruin with exponential implementation delays}

We first give a descriptive definition of the time of Parisian ruin, here denoted by $\tau_q$,  using It\^o's excursion theory (for excursions away from zero) for spectrally negative L\'evy processes. In order to do so, we \textit{mark} the Poisson point process of excursions away from zero with independent copies of a generic exponential random variable  $\mathbf{e}_q$ with parameter $q>0$. We will refer to them as implementation clocks. If  the length of the negative part of a given  excursion away from $0$ is less than its associated implementation clock, then such an excursion is neglected as far as ruin is
concerned. More precisely, we assume that  ruin occurs at the first time that an implementation
clock rings before the end of its corresponding excursion. Formally, let $G$ be the set of left-end points of negative excursions, and for each $g\in G$ consider an independent, exponentially distributed random variables $\mathbf{e}_q^{g}$, also independent of $X$. Then we define the time of Parisian ruin by
\[\tau_q=\inf\{t: X_t<0 \mbox { and } t>g_t +\mathbf{e}_q^{g_t}\} , \]
where $g_t=\sup\{s\leq t: X_s\geq 0\}$. Note that $X_t<0$ implies that $g_t\in G$.

It is worth pointing out that $\tau_q$ can be defined recursively in the case when the L\'evy insurance risk processes has paths of bounded variation, see for example \cite{LRZ-0}. We will not make any assumptions on the variation of $X$ here and our method uses a limiting argument which is motivated by the work of Loeffen et al. \cite{L}. For ease of presentation, we assume in this section that the underlying L\'evy insurance risk process $X$ satisfies the \textit{net profit condition}, i.e. 
\begin{equation}\label{net_profit_condition}
\e[X_1]=\psi^\prime(0+)> 0.
\end{equation}
Note that this assumption is actually not needed for our main results in the next section and is really only relevant to retrieve from our formulas the expression for $\mathbb{P}_x(\tau_q<\infty)$ as established in \cite{LRZ}; see Equation~\eqref{Parisruin} below.

Let $\varepsilon>0$ and consider the path of $X$ up to the first time that the process returns to $0$ after reaching the level $-\varepsilon$, More precisely, on the event $\{\tau_{-\varepsilon}^-<\infty\}$ let
\[
(X_t, 0\le t\le \tau_0^{+,\varepsilon})\qquad\textrm{where }\qquad  \tau_0^{+,\varepsilon}=\inf\{t >\tau_{-\varepsilon}^-: X_t>0\}.
\]
Let $\tau_{-\varepsilon}^{-,1}:=\tau_{-\varepsilon}^-$ and $\tau_0^{+, 1}:=\tau_0^{+,\varepsilon}$. Recursively, we define two sequences of stopping times $(\tau_{-\varepsilon}^{-,k})_{k\geq 1}$ and $(\tau_{0}^{+,k})_{k\geq 1}$ as follows: for $k\geq 2$, if
\begin{equation*}
\tau_{-\varepsilon}^{-,k}:=\inf\{t>\tau_{0}^{+,k-1}:X_t<-\varepsilon\}
\end{equation*}
is finite, define
\begin{equation*}
\tau_{0}^{+,k}=\inf\{t>\tau_{-\varepsilon}^{-,k}:X_t>0\}.
\end{equation*}
Let
\[K^\varepsilon=\inf\{k: \tau_{-\varepsilon}^{-,k}=\infty\}\] 
and denote 
$Y^{(k)}=(X_{t},  \tau_{0}^{+,k-1}\le t\le \tau_{0}^{+,k})$  for $k<K^\varepsilon$. 
We call $(Y^{(k)})_{1\leq k<K^\varepsilon}$ the $\varepsilon$-excursions of $X$ away from $0$ and note that due to the strong Markov property they are independent and identically distributed.
 Observe that under the net profit condition (\ref{net_profit_condition}) we necessarily have that $K^\varepsilon$ is almost surely finite. We also observe that the limiting case, i.e.\ when $\varepsilon$ goes to $0$, corresponds to the usual excursion of $X$ away from $0$. To avoid confusion we call the limiting case a $0$-excursion.
Note that each $\varepsilon$-excursion ends with a $0$-excursion that reaches the level $-\varepsilon$ (possibly preceeded by excursions not reaching this level). For each $k\ge 1$, we denote by $\mathbf{e}_q^k$ the implementation clock of the last $0$-excursion of $Y^{(k)}$.  

We define the approximated Parisian ruin time $\tau^{\varepsilon}_q$ as in \cite{LRZ-0} by
\begin{equation*}
\tau^{\varepsilon}_q:=\tau_{-\varepsilon}^{-,k^{\varepsilon}_q}+\mathbf{e}_q^{k^{\varepsilon}_q},\notag
\end{equation*}
where
\begin{equation*}
k_q^{\varepsilon}=\inf\{k\geq 1:\tau_{-\varepsilon}^{-,k}+\mathbf{e}_q^{k}<\tau_{0}^{+,k}\}.\notag
\end{equation*}

To see why $\tau_q^\varepsilon$ is an approximation of $\tau_q$, first note that $\tau_q^\varepsilon\geq\tau_q$. This follows from the observations that $X_s<0$ for all  $s\in(\tau_{-\varepsilon}^{-,k^{\varepsilon}_q},\tau_{-\varepsilon}^{-,k^{\varepsilon}_q}+\mathbf{e}_q^{k^{\varepsilon}_q})$ and that $\tau_{-\varepsilon}^{-,k^{\varepsilon}_q}$ is clearly greater than the left-end point of the negative excursion it is contained in. 
Furthermore, since $\lim_{\varepsilon\downarrow 0} \tau_\varepsilon^-=\tau_0^-$ $\mathbb{P}$-a.s., it readily follows that
\begin{equation}\label{tdp}
\tau_q^{\varepsilon}\xrightarrow[\varepsilon\downarrow0]{}\tau_q, \quad\text{$\mathbb{P}$-a.s. }
\end{equation}



\subsection{Main results}

In this section, we are  interested in computing different Gerber--Shiu functions for a L\'evy insurance risk process subject to Parisian ruin, as defined in the previous section. To do so, we first identify the Gerber--Shiu distribution.  It is important to point out that in all the results in this subsection the net profit condition is not necessary.

Let us now define two auxiliary functions which will appear in the Gerber--Shiu distribution. First, for $p\geq0$ and $q\in\mathbb R$ such that $p+q\geq0$ and for $x \in \mathbb{R}$, define as in \cite{LoRZ} the function
\begin{equation*}
 \mathcal{H}^{(p,q)}(x) = \mathrm e^{\Phi(p)x} \left( 1 + q\int_0^{x}  \mathrm e^{-\Phi(p)y} W^{(p+q)}(y) \mathrm dy \right) .
\end{equation*}
We further introduce, for $\theta, q\ge  0$,  $x>0$  and $y\in [-x, \infty)$, the function
\begin{equation}\label{eqg}
g(\theta, q, x, y)=W^{(\theta+q)}(x+y) - q \int_0^x W^{(\theta)}(x-z) W^{(\theta+q)}(z+y) \mathrm{d}z.
\end{equation}
Note that $g$ is of the same form as $\mathcal{W}_a^{(p,q)}$ in \cite{LoRZ}.
%

Here is the main result of this paper.
\begin{theorem}\label{main}
For $\theta, a, b \geq0$, $x \in [-a,b)$ and $y \in [-a,0]$, we have
\begin{equation}\label{main2}
\e_x \Big[ \mathrm{e}^{-\theta \parisiantime}, X_{\parisiantime} \in \mathrm{d}y , \parisiantime < \tau_b^+ \wedge \tau_{-a}^- \Big] = q\Bigg[\frac{g(\theta, q, x, a)}{g(\theta, q, b, a)}g(\theta, q, b, -y)- g(\theta, q, x, -y)\Bigg] \mathrm{d}y .
\end{equation}
\end{theorem}

Note that the above result can be written differently using the identity in Equation~\eqref{eq:scale_sym}. More precisely, one can re-write  $g(\theta, q, x, y)$ as follows:
\begin{equation}\label{main_v2}
g(\theta, q, x, y)=W^{(\theta)}(x+y)+q\int_0^y W^{(\theta)}(x+y-z) W^{(\theta+q)}(z) \mathrm{d}z.
\end{equation}

By taking appropriate limits in  Equation~\eqref{main2}, either with the definition of $g(\theta, q, x, y)$ given in (\ref{eqg}) or in (\ref{main_v2}), one can obtain the following corollary:
\begin{corollary}\label{C:norestriction}
For $\theta, a, b \geq0$, then:
\begin{enumerate}
\item for $x \geq -a$ and $y \in [-a,0]$, we have
\begin{equation}\label{cor:part_one}
\e_x \Big[ \mathrm{e}^{-\theta \parisiantime}, X_{\parisiantime} \in \mathrm{d}y , \parisiantime < \tau_{-a}^- \Big] =  q \Bigg[\frac{g(\theta, q, x, a)}{\mathcal{H}^{(\theta,q)}(a)} \mathcal{H}^{(\theta,q)}(-y) - g(\theta, q, x, -y)\Bigg] \mathrm{d}y .
\end{equation}
\item for $x \leq b$ and $y \in (-\infty,0]$, we have
\begin{equation}\label{cor:part_two}
\e_x \Big[ \mathrm{e}^{-\theta \parisiantime}, X_{\parisiantime} \in \mathrm{d}y , \parisiantime < \tau_{b}^+ \Big] =  q \Bigg[ \frac{\mathcal{H}^{(\theta+q,-q)}(x)}{\mathcal{H}^{(\theta+q,-q)}(b)} g(\theta, q, b, -y)- g(\theta, q, x, -y)\Bigg] \mathrm{d}y .
\end{equation}
\item for $x \in \reals$ and $y \in (-\infty,0]$, we have
\begin{multline}\label{cor:part_three}
\e_x \Big[ \mathrm{e}^{-\theta \parisiantime}, X_{\parisiantime} \in \mathrm{d}y , \parisiantime < \infty \Big] \\
= \Bigg[ \Big( \Phi(\theta+q)-\Phi(\theta) \Big) \mathcal{H}^{(\theta+q,-q)}(x) \mathcal{H}^{(\theta,q)}(-y) - q g(\theta, q, x, -y) \Bigg] \mathrm{d}y .
\end{multline}
\end{enumerate} 
\end{corollary}

Before moving on to the proofs of these results, let us see how we can use the Gerber--Shiu distributions of Theorem~\ref{main} and Corollary~\ref{C:norestriction} to compute specific Gerber--Shiu functions and derive a number of identities established in the literature. Consider for $\lambda \geq 0$ the Gerber--Shiu function
$$
\e_x \Big[ \mathrm{e}^{-\theta \parisiantime + \lambda X_{\parisiantime}} ; \parisiantime < \tau^+_b \Big] = \int_{-\infty}^0 \mathrm{e}^{\lambda y} \e_x \left[ \mathrm{e}^{-\theta \parisiantime}, X_{\parisiantime} \in \mathrm{d}y , \parisiantime < \tau^+_b \right].
$$

To calculate this integral, we make use of the following identity (see Equation (6) in \cite{LoRZ})
$$
(q-p)\int_0^x W^{(p)}(x-y) Z^{(q)}(y) \mathrm{d}y = Z^{(q)}(x) - Z^{(p)}(x),
$$
which holds for for $p,q \geq 0$ and $x \in \reals$.
Invoking Equation~\eqref{cor:part_two}, a  direct calculation then yields (by letting $\lambda\downarrow 0$)
that the Laplace transform of the time to ruin before the surplus exceeds the level $b$ is given by
\begin{equation}\label{extL}
\e_x \Big[ \mathrm{e}^{-\theta \parisiantime} ; \parisiantime < \tau_b^+ \Big] =\frac{q}{\theta+q}\left(Z^{(\theta)}(x) -\frac{\mathcal{H}^{(\theta+q, -q)}(x)}{\mathcal{H}^{(\theta+q, -q)}(b)}Z^{(\theta)}(b)\right).
\end{equation}

The above identity extends the result of  Landriault et al. (see Lemma 2.2 in \cite{LRZ-0}), in the case of exponential implementation delays and when the insurance risk process $X$ has paths of bounded variation.  We observe that the function $\mathcal{H}^{(\theta+q, -q)}$ is the same as the function $H^{(\theta)}_d$ defined in section 2.2 in \cite{LRZ-0}.

Next, we are interested in computing the probability of Parisian ruin in the case when the net profit condition (\ref{net_profit_condition}) is satisfied. To this end we take $\theta=0$ and let $b\rightarrow \infty$ and find
\begin{equation}\label{Parisruin}
\p_x \left( \parisiantime < \infty \right) =\lim_{b\rightarrow\infty}\left(1-\frac{\mathcal{H}^{(q,-q)}(x)}{\mathcal{H}^{(q,-q)}(b)}\right)=1 - \psi'(0+) \frac{\Phi(q)}{q} \mathcal{H}^{(q,-q)}(x)
\end{equation}
as $\lim_{b\rightarrow \infty}W(b)=1/\psi'(0)$. The probability of Parisian ruin (\ref{Parisruin}) agrees with Theorem 1 and Corollary 1 in \cite{LRZ} since we have the following identity (using a change of variable and an integration by parts):
$$
\mathcal{H}^{(q,-q)}(x) = q \int_0^{\infty} \mathrm{e}^{-\Phi(q)y} W(x+y) \mathrm{d}y .
$$

We finish this section with two remarks.
\begin{rem}
There is an interesting link with the results obtained in \cite{AIZ} concerning exit identities for a spectrally negative L\'evy process observed at Poisson arrival times. In particular, consider
\[T_0^-=\min\{T_i: X(T_i)<0\}\]
where $T_i$ are the arrival times of an independent Poisson process with rate $q$.
 By taking $\theta=0$ and integrating $e^{uy}$ ($u\geq 0$) with respect to the density given in Equation~\eqref{cor:part_two} of Corollary \ref{C:norestriction} we retrieve the same expression for 
\[\mathbb{E}_x\left[e^{uX_{\tau_q}},\tau_q<\tau_b^+\right]\] as is given in equation (15) of Theorem 3.1 in \cite{AIZ} for 
\[\mathbb{E}_x\left[e^{u X_{T_0^-}}, T_0^-<\tau_b^+\right]. \]
The method of proof in \cite{AIZ} relies mostly on the strong Markov property and fluctuation identities for spectrally negative L\'evy processes.
\end{rem}

\begin{rem} Note that since  $\{\tau_q<\tau_b^+\}=\{\overline{X}_{\tau_q}<b\}$, with $\overline{X}_t=\sup_{0\leq s\leq t}X_s$ the running supremum process, we can also derive a more general form for the Gerber--Shiu measure that takes into account the law of the process and its running supremum (as well as its running infimum) up to the time of Parisian ruin. For the sake of brevity the explicit form of this joint law is left to the reader.
\end{rem}

\section{Proofs}

\begin{proof}[Proof of Theorem~\ref{main}]
%
Take $\varepsilon\in(0,a)$. We first compute
\begin{equation}\label{ott}
\mathbb{E}\left[\mathrm e^{-\theta\tau^{\varepsilon}_q}f \left(-X_{\tau^{\varepsilon}_q} \right) \mathbf{1}_{\{\tau^{\varepsilon}_q<\tau_b^+\wedge\tau_{-a}^-\}}\right]
\end{equation}
for a bounded, continuous function $f$.
Here, we  express (\ref{ott}) in terms of the $\varepsilon$-excursions of $X$ confined to the interval $[-a,b]$ and such that the time that each $\varepsilon$-excursion away from $0$ spends below $0$ after reaching the level $-\varepsilon$ is less than its associated implementation clock, followed by the first $\varepsilon$-excursion away from $0$ that exits the interval $[-a,b]$ or such that the time that the $\varepsilon$-excursion spends below $0$ after reaching the level $-\varepsilon$ is greater than its implementation clock. More precisely, let  $(\xi_s^{i,\varepsilon}, 0\le s\le \ell^{\varepsilon}_i)$ be the  $i$-th $\varepsilon$-excursion of $X$ away from $0$ confined to the interval $[-a,b]$ and such that $\ell^{\varepsilon}_i-\sigma_{-\varepsilon}^i\leq \mathbf{e}_q^i$, where $\ell^{\varepsilon}_i$ denotes the length of  $\xi^{i,\varepsilon}$, and 
\[
\sigma_{-\varepsilon}^i=\inf\{s<\ell^{\varepsilon}_i:\xi_s^{i,\varepsilon}<-\varepsilon\}.
\] 
Similarly, let $(\xi^{*,\varepsilon}_s, 0\le s\le \ell^{\varepsilon}_*)$  be the first $\varepsilon$-excursion  of $X$  away from $0$   that exits the interval $[-a,b]$, or such that $\ell^{\varepsilon}_*-\sigma_{-\varepsilon}^*>\mathbf{e}_q^{k_q}$ where $\ell^{\varepsilon}_*$  is its length and 
\[
\sigma_{-\varepsilon}^*=\inf\{s<\ell^{\varepsilon}_*:\xi_s^{*,\varepsilon}<-\varepsilon\}.
\] 
We also define the infimum and supremum of the  excursion $\xi^{*,\varepsilon}$, as follows
\[
\underline{\xi}^{*,\varepsilon}=\inf_{s<\ell^{\varepsilon}_*}\xi^{*,\varepsilon}_s \qquad\textrm{ and }\qquad \overline{\xi}^{*,\varepsilon}=\sup_{s<\ell^{\varepsilon}_*}\xi^{*,\varepsilon}_s.
\]
From the strong Markov property, it is clear that the random variables $\left( \mathrm e^{-q\ell^{\varepsilon}_i} \right)_{i \geq 1}$  are i.i.d. and also independent of 
\[
\Xi^{(*, \varepsilon)}_{a,b}:= \mathrm e^{-\theta(\sigma_{-\varepsilon}^*+\mathbf{e}_q^{k_q})}f\Big(-\xi^{*,\varepsilon}_{\sigma_{-\varepsilon}^*+\mathbf{e}_q^{k_q}}\Big)\mathbf{1}_{\{\ell^{\varepsilon}_*<\infty\}}\mathbf{1}_{\{\overline{\xi}^{*,\varepsilon}\leq b\}}\mathbf{1}_{\{\underline{\xi}^{*,\varepsilon}\geq -a\}}.
\]
Let $\zeta = \tau_0^{+,\varepsilon}$ and $p=\p(E)$, where
\[
E=\Big\{\sup_{t\leq \zeta} X_t\leq b,\inf_{t\leq \zeta}X_t\geq -a,  \zeta-\tau_{-\varepsilon}^-\leq \mathbf{e}_q\Big\}.
\] 
A standard description of $\varepsilon$-excursions of $X$ away from $0$ confined to the interval $[-a,b]$ with the amount of time spent below $0$ after reaching the level $-\varepsilon$ less than an exponential time, dictates that the number of such $\varepsilon$-excursions is distributed according to an independent geometric random variable, say $G_p$, (supported on $\{0,1,2,\ldots\}$) with parameter $p$. Moreover, the random variables $\left( \mathrm e^{-q\ell^{\varepsilon}_i} \right)_{i \geq 1}$ have the same distribution as $\mathrm e^{-\theta\zeta}$ under the conditional law $\p(\cdot|E)$ and the random variable $\Xi^{(*, \varepsilon)}_{a,b}$ is equal in distribution to 
\[
\mathrm e^{-\theta (\tau_{-\varepsilon}^-+\mathbf{e}_q)} f \left( -X_{\tau_{-\varepsilon}^- +\mathbf{e}_q} \right) \mathbf{1}_{\{\inf_{t\leq \tau_{-\varepsilon}^- +\mathbf{e}_q} X_t \geq -a\}}\mathbf{1}_{\{\sup_{t\leq \tau_{-\varepsilon}^- +\mathbf{e}_q} X_t\leq b\}},
\]
but now under the conditional law $\p(\cdot|E^c)$. Then, it follows that
\begin{equation}\label{putin}
\begin{split}
\mathbb{E}\bigg[ \mathrm e^{-\theta\tau_q^{\varepsilon}}&f \left(-X_{\tau_q^{\varepsilon}} \right) \mathbf{1}_{\{\tau^{\varepsilon}_q<\tau_b^+\wedge\tau_{-a}^-\}}\bigg]\\
&=\e\left[\prod_{i=0}^{G_p} \mathrm e^{-\theta\ell^{\varepsilon}_i} \mathrm e^{-\theta(\sigma_{-\varepsilon}^*+\mathbf{e}_q^{k_q})} f\Big(-\xi^{*,\varepsilon}_{\sigma_{-\varepsilon}^*+\mathbf{e}_q^{k_q}}\Big)\mathbf{1}_{\{\ell^{\varepsilon}_*<\infty\}}\mathbf{1}_{\{\overline{\xi}^{*,\varepsilon}\leq b\}}\mathbf{1}_{\{\underline{\xi}^{*,\varepsilon}\geq -a\}} \right]\\
&=\mathbb{E}\left[\e\left[ \mathrm e^{-\theta\ell^{\varepsilon}_1}\right]^{G_p}\right]\e\left[ \mathrm e^{-\theta(\sigma_{-\varepsilon}^*+\mathbf{e}_q^{k_q})}f\Big(-\xi^{*,\varepsilon}_{\sigma_{-\varepsilon}^*+\mathbf{e}_q^{k_q}}\Big)\mathbf{1}_{\{\ell^{\varepsilon}_*<\infty\}}\mathbf{1}_{\{\overline{\xi}^{*,\varepsilon}\leq b\}}\mathbf{1}_{\{\underline{\xi}^{*,\varepsilon}\geq -a\}} \right] .
\end{split}
\end{equation}
Recall that the moment generating function $F$ of the  geometric random variable $G_p$ satisfies
\[
F(s)=\frac{\overline{p}}{1-ps},\quad |s|<\frac{1}{p},
\]
where $\overline{p}=1-p$.
Therefore, if we can make sure that $\e\left[\mathrm e^{-\theta\ell^{\varepsilon}_1}\right] < 1/p$, then
\begin{equation}\label{pieces together}
\mathbb{E}\left[\e\left[\mathrm e^{-\theta\ell^{\varepsilon}_1}\right]^{G_p}\right] =\displaystyle\frac{\overline{p}}{1-p\e\left[\mathrm e^{-\theta\ell^{\varepsilon}_1}\right]} .
\end{equation}
Now, using \eqref{putin} and \eqref{pieces together}, we have
\begin{equation}\label{it}
\mathbb{E}\bigg[ \mathrm e^{-\theta \tau_q^{\varepsilon}} f \left(-X_{\tau_q^{\varepsilon}} \right) \mathbf{1}_{\{\tau^{\varepsilon}_q<\tau_b^+\wedge\tau_{-a}^-\}} \bigg] = \frac{\overline{p} \e \left[ \Xi^{(\ast,\varepsilon)}_{a,b} \right]}{1 - p \e \left[ \mathrm{e}^{-\theta \ell_1^\varepsilon} \right]} .
\end{equation}

Taking account of the remarks in the previous paragraph and making use of the strong Markov property, we  have
\begin{align*}
\e\left[\mathrm e^{-\theta\ell^{\varepsilon}_1}\right]&=\frac{1}{p}\e\left[\mathrm e^{-\theta\tau_{-\varepsilon}^-} \mathbf{1}_{\{\tau_{-\varepsilon}^-<\tau_b^+ \wedge \tau_{-a}^-\}} \e_{X_{\tau_{-\varepsilon}^-}}\left[ \mathrm e^{-(\theta+q)\tau_0^+};\tau_0^+<\tau_{-a}^-\right]\right]\notag\\
&=\frac{1}{p} \e_{\varepsilon}\left[\mathrm e^{-\theta\tau_{0}^-}\mathbf{1}_{\{\tau_{0}^-<\tau_{b+\varepsilon}^+\}} \frac{W^{(\theta+q)}(X_{\tau_{0}^-}-\varepsilon+a)}{W^{(\theta+q)}(a)}\right]\notag .
\end{align*}
Note that we do not need the indicator function of $\{X_{\tau_0^-}-\varepsilon>-a\}$ since, on its complement, the scale function vanishes. Note also that it is now clear from the above computation that $\e\left[\mathrm e^{-\theta\ell^{\varepsilon}_1}\right] < 1/p$. Using the identity in Equation~\eqref{eq:scale_sym}, one can write
\begin{multline*}
\e_\varepsilon \left[ \mathrm{e}^{-\theta \tau_0^-} W^{(\theta+q)}\left( X_{\tau_0^-}-\varepsilon+a\right) \ind_{\{\tau_0^- < \tau_{b+\varepsilon}^+\}} \right] \\
= W^{(\theta+q)}(a) - q \int_{a-\varepsilon}^a W^{(\theta)}(a-z) W^{(\theta+q)}(z) \mathrm{d}z \\
- \frac{W^{(\theta)}(\varepsilon)}{W^{(\theta)}(b+\varepsilon)} \left( W^{(\theta+q)}(b+a) - q \int_{a-\varepsilon}^{b+a} W^{(\theta)}(b+a-z) W^{(\theta+q)}(z) \mathrm{d}z \right) .
\end{multline*}
As a consequence,
\begin{multline*}
1 - p \e \left[ \mathrm{e}^{-\theta \ell_1^\varepsilon} \right] = \frac{q}{W^{(\theta+q)}(a)} \int_{a-\varepsilon}^a W^{(\theta)}(a-z) W^{(\theta+q)}(z) \mathrm{d}z \\
+ \frac{W^{(\theta)}(\varepsilon)}{W^{(\theta+q)}(a) W^{(\theta)}(b+\varepsilon)} \left( W^{(\theta+q)}(b+a) - q \int_{a-\varepsilon}^{b+a} W^{(\theta)}(b+a-z) W^{(\theta+q)}(z) \mathrm{d}z \right) .
\end{multline*}

Next, we compute  the Laplace transform of $\Xi^{(*, \varepsilon)}_{a,b}$. Recalling that under $\p(\cdot|E^c)$ and on the event  $\{\overline{\xi}^{*,\varepsilon}<b,\underline{\xi}^{*,\varepsilon}\geq -a \}$, we necessarily have that the excursion goes below the level $-\varepsilon$ and the exponential clock rings before the end of the excursion, i.e.
\begin{multline}\label{bigshit}
\overline{p} \e \left[ \Xi^{(\ast,\varepsilon)}_{a,b} \right] = \e \left[ \mathrm{e}^{-\theta \tau_{-\varepsilon}^-} \e_{X_{\tau_{-\varepsilon}^-}} \left[ \mathrm{e}^{-\theta \exptime} f \left( -X_{\exptime} \right) ; \exptime < \tau_{-a}^- \wedge \tau_0^+ \right]  \ind_{\{\tau_{-\varepsilon}^- < \tau_{-a}^- \wedge \tau_b^+\}} \right] \\
= q \int_{-a}^0 f(-y) \e \left[ \mathrm{e}^{-\theta \tau_{-\varepsilon}^-} \left\{\frac{W^{(\theta+q)}(X_{\tau_{-\varepsilon}^-}+a)W^{(\theta+q)}(-y)}{W^{(\theta+q)}(a)}-W^{(\theta+q)}(X_{\tau_{-\varepsilon}^-}-y) \right\}  \ind_{\{\tau_{-\varepsilon}^- < \tau_b^+\}} \right] \mathrm{d}y ,
\end{multline}
thanks to Fubini's theorem and identity~\eqref{fi4} in Theorem~\ref{fi}. Using once more the identity in Equation~\eqref{eq:exp_scale} and rearranging the terms, one can write
\begin{multline*}
\e \left[ \mathrm{e}^{-\theta \tau_{-\varepsilon}^-} \left\{\frac{W^{(\theta+q)}(X_{\tau_{-\varepsilon}^-}+a)W^{(\theta+q)}(-y)}{W^{(\theta+q)}(a)}-W^{(\theta+q)}(X_{\tau_{-\varepsilon}^-}-y) \right\}  \ind_{\{\tau_{-\varepsilon}^- < \tau_b^+\}} \right] \\
= \frac{W^{(\theta)}(\varepsilon)}{W^{(\theta)}(b+\varepsilon)} \left\lbrace \left[ W^{(\theta+q)}(b-y) - q \int_0^{b+\varepsilon} W^{(\theta)}(b+\varepsilon-z) W^{(\theta+q)}(z-y-\varepsilon) \mathrm{d}z \right] \right. \\
\qquad - \frac{W^{(\theta+q)}(-y)}{W^{(\theta+q)}(a)} \left. \left[ W^{(\theta+q)}(b+a) - q \int_0^{b+\varepsilon} W^{(\theta)}(b+\varepsilon-z) W^{(\theta+q)}(z+a-\varepsilon) \mathrm{d}z \right] \right\rbrace \\
+ q \left\lbrace \int_0^\varepsilon W^{(\theta)}(\varepsilon-z) W^{(\theta+q)}(z-y-\varepsilon) \mathrm{d}z \right. \\
- \left. \frac{W^{(\theta+q)}(-y)}{W^{(\theta+q)}(a)} \int_0^\varepsilon W^{(\theta)}(\varepsilon-z) W^{(\theta+q)}(z+a-\varepsilon) \mathrm{d}z \right\rbrace .
\end{multline*}

Now we are interested in computing the limit of $\mathbb{E}\bigg[ \mathrm e^{-\theta \tau_q^{\varepsilon}} f \left(-X_{\tau_q^{\varepsilon}} \right) \mathbf{1}_{\{\tau^{\varepsilon}_q<\tau_b^+\wedge\tau_{-a}^-\}} \bigg]$, as given in Equation~\eqref{it}, when $\varepsilon$ goes to $0$. We use the above computations for the numerator and the denominator, and we divide both by $W^{(\theta)}(\varepsilon)$. First, we have
\begin{multline*}
\frac{1 - p \e \left[ \mathrm{e}^{-\theta \ell_1^\varepsilon} \right]}{W^{(\theta)}(\varepsilon)} = \frac{q}{W^{(\theta+q)}(a)} \frac{\int_{a-\varepsilon}^a W^{(\theta)}(a-z) W^{(\theta+q)}(z) \mathrm{d}z}{W^{(\theta)}(\varepsilon)} \\
+ \frac{1}{W^{(\theta+q)}(a) W^{(\theta)}(b+\varepsilon)} \left( W^{\theta+q)}(b+a) - q \int_{a-\varepsilon}^{b+a} W^{(\theta)}(b+a-z) W^{(\theta+q)}(z) \mathrm{d}z \right) \\
\underset{\varepsilon \downarrow 0}{\longrightarrow} \frac{1}{W^{(\theta+q)}(a) W^{(\theta)}(b)} \left( W^{(\theta+q)}(b+a) - q \int_{a}^{b+a} W^{(\theta)}(b+a-z) W^{(\theta+q)}(z) \mathrm{d}z \right) .
\end{multline*}
Indeed, when the process has paths of bounded variation, we have
$$
\frac{\int_{a-\varepsilon}^a W^{(\theta)}(a-z) W^{(\theta+q)}(z) \mathrm{d}z}{W^{(\theta)}(\varepsilon)} \underset{\varepsilon \downarrow 0}{\longrightarrow} \frac{0}{W^{(\theta)}(0)} = 0 ,
$$
while, when it has paths of unbounded variation, we have
$$
\frac{1}{W^{(\theta)}(\varepsilon)/\varepsilon} \frac{\int_{a-\varepsilon}^a W^{(\theta)}(a-z) W^{(\theta+q)}(z) \mathrm{d}z}{\varepsilon} \underset{\varepsilon \downarrow 0}{\longrightarrow} \frac{W^{(\theta)}(0) W^{(\theta+q)}(a)}{W^{(\theta)\prime}(0)} = 0 .
$$

Similarly, using Lebesgue's dominated convergence theorem, we have
\begin{multline*}
\frac{\overline{p} \e \left[ \Xi^{(\ast,\varepsilon)}_{a,b} \right] }{W^{(\theta)}(\varepsilon)} \underset{\varepsilon \downarrow 0}{\longrightarrow} q \int_{-a}^0 \frac{f(-y)}{W^{(\theta)}(b)} \left\lbrace \left[ W^{(\theta+q)}(b-y) - q \int_0^{b} W^{(\theta)}(b-z) W^{(\theta+q)}(z-y) \mathrm{d}z \right] \right. \\
\qquad - \frac{W^{(\theta+q)}(-y)}{W^{(\theta+q)}(a)} \left. \left[ W^{(\theta+q)}(b+a) - q \int_0^{b} W^{(\theta)}(b-z) W^{(\theta+q)}(z+a) \mathrm{d}z \right] \right\rbrace \mathrm{d}y.
\end{multline*}

Putting all the pieces together, we deduce
\begin{multline}\label{con}
\lim_{\varepsilon\downarrow0} \mathbb{E}\bigg[\mathrm e^{-\theta \tau^{\varepsilon}_q} f \left(-X_{\tau^{\varepsilon}_q} \right) \mathbf{1}_{\{\tau^{\varepsilon}_q<\tau_b^+\wedge\tau_{-a}^-\}}\bigg] \\
= q \int_{-a}^0 f(-y) \left\lbrace W^{(\theta+q)}(a) \frac{g(\theta, q, b, -y)}{g(\theta, q, b, a)} - W^{(\theta+q)}(-y) \right\rbrace \mathrm{d}y ,
\end{multline}
where $g(\theta, q, x, y)$ is given as in (\ref{eqg}).

Hence, from (\ref{tdp}) we have that if $f$ is a continuous and bounded function, we can use Lebesgue's dominated convergence theorem to conclude 
$$
\mathbb{E}\bigg[\mathrm e^{-\theta \tau_q} f \left(-X_{\tau_q} \right) \mathbf{1}_{\{\tau_q<\tau_b^+\wedge\tau_{-a}^-\}}\bigg] = \lim_{\varepsilon\downarrow0} \mathbb{E}\bigg[\mathrm e^{-\theta \tau^{\varepsilon}_q} f \left(-X_{\tau^{\varepsilon}_q} \right) \mathbf{1}_{\{\tau^{\varepsilon}_q<\tau_b^+\wedge\tau_{-a}^-\}}\bigg] .
$$

In order to prove the result when the process starts at $x>0$, we consider the first $0$-excursion. Here, we have two possibilities when the process $X$ goes below the level $0$, either it touches $0$ (coming from below) before the exponential clock rings, or the clock rings before the process $X$ finishes its negative excursion. In the first case, once the process $X$ returns to $0$, we can start the procedure all over again. Hence, using the strong Markov property and the independence between the excursions, we obtain 
\begin{align*}
\e_x & \bigg[\mathrm e^{-\theta \tau_q} f \left(-X_{\tau_q} \right) \mathbf{1}_{\{\tau_q<\tau_b^+\wedge\tau_{-a}^-\}}\bigg] \\
&= \e_x \left[\mathrm e^{-\theta \tau_0^-} \e_{X_{\tau_0^-}} \Big[\mathrm e^{-\theta \mathbf{e}_q} f \left( -X_{\mathbf{e}_q} \right) \mathbf{1}_{\{\mathbf{e}_q < \tau_{-a}^- \wedge \tau_0^+\}} \Big] \mathbf{1}_{\{\tau_0^-<\tau_b^+ \wedge \tau_{-a}^-\}} \right] \\
& \qquad + \e_x \left[\mathrm e^{-\theta \tau_0^-} \e_{X_{\tau_0^-}} \Big[\mathrm e^{-\theta \tau_0^+} ; \mathbf{1}_{\{\tau_0^+ < \tau_{-a}^- \wedge \mathbf{e}_q\}} \Big] \mathbf{1}_{\{\tau_0^-<\tau_b^+\}} \right] \e_0 \bigg[\mathrm e^{-\theta \tau_q} f \left(-X_{\tau_q} \right) \mathbf{1}_{\{\tau_q<\tau_b^+\wedge\tau_{-a}^-\}}\bigg] .
\end{align*}
Using once again the identities in Equations~\eqref{fi2}, \eqref{fi4} and~\eqref{eq:exp_scale}, and putting all the pieces together yield the result.
\end{proof}

\begin{proof}[Proof of Corollary~\ref{C:norestriction}]
The first two results in Equation~\eqref{cor:part_one} and Equation~\eqref{cor:part_two} follow by taking appropriate limits, i.e.\ letting $a$ and $b$ go to infinity in Equation~\eqref{main2} or Equation~\eqref{main_v2}, and by using the following identity (see e.g.\ Exercice 8.5 in \cite{K}): for  $r\geq 0$ and $x \in \reals$,
$$
\lim_{c \to \infty} \frac{W^{(r)}(c-x)}{W^{(r)}(c)} = \mathrm e^{-\Phi(r) x} .
$$
The third part of the Corollary, i.e.\ Equation~\eqref{cor:part_three}, is obtained by computing the following limit
$$
\lim_{b \to \infty} \e_x \left[ \mathrm{e}^{-\theta \parisiantime}, X_{\parisiantime} \in \mathrm{d}y , \parisiantime < \tau_b^+ \right] = \e_x \left[ \mathrm{e}^{-\theta \parisiantime}, X_{\parisiantime} \in \mathrm{d}y , \parisiantime < \infty \right] ,
$$
and by observing that
\begin{equation}\label{lim1}
\lim_{b \to \infty} \frac{W^{(\theta)}(b)}{\mathrm{e}^{\Phi(\theta+q)b} - q \int_0^b W^{(\theta)}(b-z) \mathrm{e}^{\Phi(\theta+q)z} \mathrm{d}z} = \frac{\Phi(\theta+q)-\Phi(\theta)}{q}
\end{equation}
and
\begin{multline}
\lim_{b \to \infty} \frac{W^{(\theta+q)}(b-y) - q \int_0^b W^{(\theta)}(b-z) W^{(\theta+q)}(z-y) \mathrm{d}z}{\mathrm{e}^{\Phi(\theta+q)b} - q \int_0^b W^{(\theta)}(b-z) \mathrm{e}^{\Phi(\theta+q)z} \mathrm{d}z} \\
= \frac{\Phi(\theta+q)-\Phi(\theta)}{q} \left( \mathrm{e}^{-\Phi(\theta)y} + q \int_0^{-y} \mathrm{e}^{-\Phi(\theta)(y+z)} W^{(\theta+q)}(z) \mathrm{d}z \right) .\label{lim2}
\end{multline}
Here, (\ref{lim1}) follows from an application of l'H\^opital's rule to the quotient
\[\frac{e^{-\Phi(\theta+q)b}W^{(\theta)}(b)}{1-q\int_0^b e^{-\Phi(\theta+q)z}W^{(\theta)}(z)\mathrm{d}z}
\]
and
\[\lim_{b\rightarrow\infty} \frac{W^{(\theta)^\prime}(b)}{W^{(\theta)}(b)}=\Phi(\theta),\]
whereas (\ref{lim2}) can be obtained by combining (\ref{lim1}) with

\[q\int_0^b W^{(\theta)}(b-z) W^{(\theta+q)}(z-y) \mathrm{d}z=W^{(\theta+q)}(b-y)-W^{(\theta)}(b-y)-q\int_0^{-y}W^{(\theta)}(b-z-y)W^{(\theta+q)}(z)\mathrm{d}z,\]
with the latter due to (\ref{eq:scale_sym}).

\end{proof}


\section*{Acknowledgments}
E.J. Baurdoux was visiting CIMAT, Guanajuato when part of this work was carried out and he is grateful for their hospitality and support.

\end{document}